\documentclass[reqno,a4paper,12pt]{amsart}
\title{On the
topological 4-genus of torus knots}
\author{S.~Baader, P.~Feller, L.~Lewark, L.~Liechti}
\thanks{The second author and the third and fourth author gratefully acknowledge support by the SNSF grants
155477 and 159208, respectively. The third author thanks the EPSRC grant EP/K00591X/1 for providing computing facilities.}
\usepackage[latin1]{inputenc}
\usepackage{graphicx,amsmath,amssymb,microtype,bbm,tabu,xcolor,hyperref,caption}
\usepackage[nameinlink,nosort]{cleveref}
\usepackage[pdftex,all,knot]{xy}
\crefname{subsection}{subsection}{subsections}
\Crefname{subsection}{Subsection}{Subsections}
\crefname{equation}{}{}
\let\cref\Cref
\hypersetup{colorlinks={true},linkcolor={black},citecolor={black},filecolor={black},urlcolor={black},
pdfauthor={\authors},pdftitle={\shorttitle}}
\newcommand{\myemail}[1]{\texttt{\href{mailto:#1}{#1}}}

\newcommand{\qua}{\hskip 0.4em \ignorespaces}
\def\arxiv#1{\relax\ifhmode\unskip\qua\fi
\href{http://arxiv.org/abs/#1}%
{\tt arXiv:\penalty -100\unskip#1}}    

\def\MR#1{\relax\ifhmode\unskip\qua\fi
\href{http://www.ams.org/mathscinet-getitem?mr=#1}{\tt MR#1}}
\def\xox#1{\csname xx#1\endcsname}

\theoremstyle{plain}
\newtheorem{theorem}{Theorem}
\newtheorem{prop}[theorem]{Proposition}
\newtheorem{lemma}[theorem]{Lemma}
\theoremstyle{remark}
\newtheorem{example}[theorem]{Example}
\newtheorem{remark}[theorem]{Remark}

\DeclareMathOperator{\rk}{rk}

\def\N{{\mathbb N}}

\begin{document}
\begin{abstract}
We prove that the topological locally flat slice genus of large torus knots takes up less than three quarters of the ordinary genus. As an application, we derive the best possible linear estimate of the topological slice genus for torus knots with non-maximal signature invariant.

\end{abstract}
\maketitle
\section{Introduction}
\label{sec1}
The Thom conjecture asserts that algebraic curves in $\mathbb{C}\mathbb{P}^2$ are genus-minimising within their homology class~\cite{KM}. More precisely, no smooth embedded surface in $\mathbb{C}\mathbb{P}^2$ has smaller genus than an  algebraic curve homologous to that surface. Regularity plays an important role here. In fact, Rudolph proved the existence of topological locally flat surfaces with strictly smaller genus than all algebraic curves homologous to it~\cite{Ru2}. A precise quantitative measure of the drop in genus for locally flat surfaces was given in~\cite{LW}. The knot theoretic version of the Thom conjecture asserts that the smooth slice genus of a positive braid knot coincides with the ordinary genus~\cite{Ru3}. Much less is known about the topological locally flat slice genus $g_4$ of positive braid knots, or even torus knots.
Positive braid knots have non-zero signature invariant $\sigma$~\cite{Ru1}, whence $g_4>0$, by the following signature bound: $|\sigma| \leq 2g_4$. This bound was proven smoothly in~\cite{murasugi}, and  for the locally flat slice genus in~\cite{KT}. Using the existence of quasipositive knots with Alexander polynomial $1$, Rudolph showed that the torus knot $T(5,6)$ has $g_4<g$, where $g$ is the classical minimal genus of knots~\cite{Ru2}. The main purpose of this paper is to show that the genus defect $\Delta g=g-g_4$ takes up a large portion of the genus for most torus knots.

\begin{theorem} \label{T1} Let $K=T(p,q)$ be a torus knot with non-maximal signature invariant, i.e.~$K \neq T(2,n)$, $T(3,4)$, $T(3,5)$. Then
$$g_4(K) \leq \frac{6}{7} \, g(K).$$
\end{theorem}

This result is sharp, since the torus knot $T(3,8)$ has $g_4=6$ and $g=7$. However, a larger genus defect is attained for torus knots with large parameters $p,q \in \N$. The classical genus formula $g(T(p,q))=\frac{1}{2}(p-1)(q-1)$ yields
$$\lim_{p,q \to \infty} \frac{2}{pq} g(T(p,q))=1.$$
Here the limit is understood as $\lim \min\{p,q\} \to \infty$ (i.e. both parameters must be taken to infinity). As we will see, the corresponding limit for $g_4$ drops by at least one
quarter. The existence of this limit follows from the subadditivity of the function $g_4(T(p,q))$ in both parameters (see the proof of Proposition~9 in the Appendix of \cite{Li} for the one-variable case known as Fekete's Lemma; the two-variable case follows from an analogous estimate between the ratios $\frac{g_4(T(p,q))}{pq}$ and  $\frac{g_4(T(N,N))}{N^2}$, where $p=aN+b$ and $q=cN+d$).

\begin{theorem} \label{T2} $$\lim_{p,q \to \infty} \frac{2}{pq} g_4(T(p,q)) < \frac{3}{4}.$$
\end{theorem}

To the best of our knowledge, no attempt at determining the actual limit has been made so far.
The signature bound $|\sigma| \leq 2g_4$ potentially allows a drop down to one-half, since
$$\lim_{p,q \to \infty} \frac{1}{pq} \sigma(T(p,q))=\frac{1}{2}.$$
The latter is an easy consequence of the signature formula for torus knots by Gordon, Litherland and Murasugi~\cite{GLM}.

We will prove \cref{T1,T2} in \cref{sec4,sec3}, respectively. The reason for the reverse order is simple: \cref{T2} implies \cref{T1}, up to finitely many values of the braid index $\min\{p,q\}$, since $\frac{3}{4}<\frac{6}{7}$.  The main tool for proving \cref{T2} is a homological improvement of Rudolph's method, which we will explain in \cref{sec2}.

The strength of this method is demonstrated in \cref{tree-limit}, which provides a sharp estimate of the topological slice genus for positive fibred arborescent links. In particular, we find prime positive braid links of arbitrarily large genus with $\frac{g_4}{g}=\frac{1}{2}$.

\medskip
\paragraph{\emph{Acknowledgements:}}
We warmly thank Maciej Borodzik and Filip Misev for helpful comments and inspiring discussions.

\section{Construction of locally flat surfaces}
\label{sec2}
Let us first briefly fix notation and conventions.
We assume all Seifert surfaces to be connected.
The genus $g(L)$ and Betti number $b_1(L)$ of a link $L$ are the minimal genus and Betti number of a Seifert surface of $L$, respectively.
Homology groups are considered over the integers.
The topological slice genus $g_4(L)$ is the minimal genus of a \emph{slice surface} of $L$, i.e.~of a
connected oriented compact surface, properly and locally flatly embedded into the 4-ball, whose boundary is $L$.
For any surface $\Sigma$, a subsurface $\Sigma' \subset \Sigma$ is simply a surface contained in $\Sigma$,
assuming neither that $\Sigma'$ is connected, nor that it is embedded properly into $\Sigma$.
We write $a_1, \ldots, a_{n-1}$ for the standard generators of the braid group on $n$ strands.
If a braid is given by a braid word $\beta$, we write $\widehat{\beta}$ for its closure.
A non-split braid word $\beta$ yields a canonical Seifert surface for $\widehat{\beta}$, which we denote
by $\Sigma(\beta)$. If, in addition, $\beta$ is positive, $\Sigma(\beta)$
is in fact the fibre surface of $\widehat{\beta}$ \cite{stallings}.
The Alexander polynomial of a bilinear integral form represented by a matrix $M$
is $\det (t\cdot M - M^{\top})\in\mathbb{Z}[t^{\pm}]$; this does not depend on the chosen matrix, and
is considered up to multiplication with a unit.

Our main tool uses Freedman's celebrated result \cite{freedman,fq} to construct
slice surfaces of lower genus from Seifert surfaces by ambient surgery.
See \cite{Fe,BaLe,FeMc} for other applications of this method.
Here we prove a version for multi-component links.
\begin{prop}\label{maintool}
Let $L$ be a link with a Seifert surface $\Sigma$. %
Let $V\subset H_1(\Sigma)$ be a subgroup. %
If the Seifert form of $\Sigma$ restricted to $V$ has Alexander polynomial $1$,
then $L$ has a slice surface of genus $g(\Sigma) - \rk V / 2$.
\end{prop}
We will call such a subgroup $V$ \emph{Alexander-trivial}.
Before the proof, let us show a sample application.
\begin{figure}[ht]
\includegraphics[width=\textwidth]{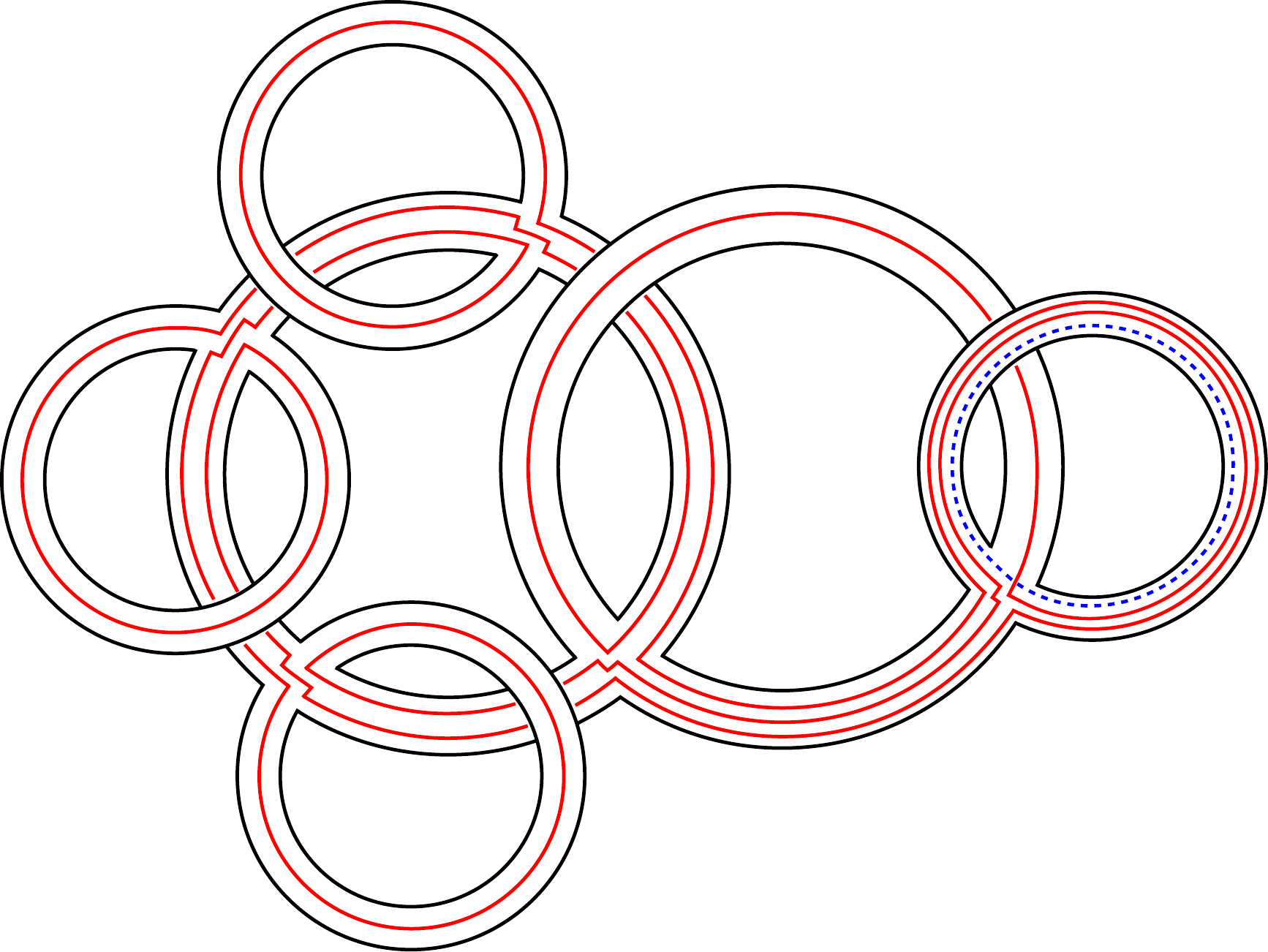}
\caption{The surface $\widetilde{X}$, which is a plumbing of six positive Hopf bands.
    For simplicity, the twists of the individual Hopf bands are not drawn.
Two curves representing homology classes of interest are drawn in red and dashed blue --
see \cref{tildeX} for details.}
\label{figX}
\end{figure}
\begin{example}
\label{tildeX}
The link $L$ given as the closure of the positive $4$-braid $a_1a_3a_2^2a_1a_3a_2^3$ has topological slice genus one. Calculating the signature yields $1=\frac{|\sigma(L)|-2}{2}\leq g_4(L)$ by the bound provided in~\cite{KT}. We now show that $L$ has a genus one slice surface. For this,
let $\widetilde X$ be the canonical fibre surface $\Sigma(a_1a_3a_2^2a_1a_3a_2^3)$.
Observe that $\widetilde X$ is a plumbing of six positive Hopf bands along an X-shaped tree,
as shown in \cref{figX}.  Here we use the fact that fibre surfaces with the same boundary link are isotopic.
The two additional simple closed curves in the figure (red and dashed blue) represent homology classes $[\gamma_1]$ and $[\gamma_2]$ in $H_1(\widetilde X)$.
We claim that the subspace $V$ generated by $[\gamma_1]$ and $[\gamma_2]$ is Alexander-trivial.
A matrix for the Seifert form of the boundary link is given by the $6\times6$ matrix $A$, where
$$A_{ii}=A_{12}=A_{23}=A_{43}=A_{53}=A_{63}=1$$
and $A_{ij} = 0$ otherwise.
In the chosen basis, $[\gamma_1]$ and $[\gamma_2]$ are represented by the vectors $(0,1,-2,1,1,1)^{\top}$ and $(1,0,0,0,0,0)^{\top}$, respectively.
A direct computation yields
\begin{align*}
 [\gamma_1]^{\top}A[\gamma_1] &= [\gamma_1]^{\top}A[\gamma_2] = 0,\\
 [\gamma_2]^{\top}A[\gamma_2] &= [\gamma_2]^{\top}A[\gamma_1] = 1,
\end{align*}
so a matrix $B$ of the Seifert form restricted to $V$ is given by
$$B=
\begin{pmatrix}
 0 & 0\\
 1 & 1
\end{pmatrix}.$$
Indeed, we now have $\det(t\cdot B - B^{\top})=t$,
which is a unit in $\mathbb{Z}[t^{\pm 1}]$. Thus, by \cref{maintool}, the boundary link $\partial\widetilde X$ possesses a slice surface of genus $g(\widetilde X) -1=1$.
Geometrically, what happens if we apply \cref{maintool} is the following:
starting from $\widetilde X$ we cut out the punctured torus $T$ defined by the union of the thickened red and dashed blue curves.
Then, using Freedman's disc theorem, we reglue a disc whose interior lies in the $4$-dimensional unit ball along $\partial T$, obtaining a slice surface with smaller genus than $\widetilde X$.
For this we use that $\partial T$ is a knot with Alexander polynomial $1$.
\end{example}
Let us now turn to the proof of \cref{maintool}.
A crucial ingredient is the following fact about the mapping class group of surfaces,
which is well-known for surfaces with at most one boundary component (see e.g.~\cite{BeMa}).
\begin{lemma}\label{mpc}
Let $\Sigma$ be a connected oriented compact surface of genus $g$ with $n$ boundary components.
An automorphism $\varphi$ of $H_1(\Sigma)$ is induced by an orientation-preserving diffeomorphism $\widetilde{\varphi}$ of $\Sigma$
if and only if $\varphi$ preserves the intersection form of $\Sigma$ and permutes the homology classes
of the boundary curves. %

\end{lemma}
\begin{proof}
Clearly every orientation-preserving diffeomorphism preserves the intersection form and maps boundary curves to boundary curves (preserving the orientations), which induces a permutation of the corresponding homology classes. %

Now let us prove that the conditions are sufficient.
Let
\[
\gamma_{1,1}, \gamma_{2,1}, \ldots, \gamma_{g,1}, \gamma_{1,2}, \ldots, \gamma_{g,2}, \delta_1, \ldots, \delta_{n-1}
\]
be a \emph{geometric basis} on $\Sigma$ (a term taken from \cite{TeGa}); that is, the $\delta_i$ are
boundary curves, $\gamma_{i,j}$ intersects $\gamma_{i,3-j}$ once (geometrically), and there are
no other geometric intersections between any of these curves.
The homology classes of these curves then form a basis of $H_1(\Sigma)$.

One easily finds a simple closed curve $\zeta\subset\Sigma$ with the following properties:
it intersects $\gamma_{1,1}$ once, does not intersect any other curve in the geometric basis,
and $[\zeta] = [\gamma_{1,2}] + [\delta_1]$. So the only basis curve affected by a Dehn twist along $\zeta$ is
$\gamma_{1,1}$, whose homology class is sent to $[\gamma_{1,1}] + [\gamma_{1,2}] + [\delta_1]$. Composing with another Dehn
twist along $\gamma_{1,2}$, one finds a diffeomorphism that sends $[\gamma_{1,1}]\mapsto [\gamma_{1,1}] + [\delta_1]$.
Similarly, for all $i\in\{1,\ldots,g\}, j\in\{1,2\}, k\in\{1,\ldots, n-1\}$, there is a diffeomorphism
sending $[\gamma_{i,j}]\mapsto [\gamma_{i,j}] + [\delta_k]$. Composing these diffeomorphisms, one may realise automorphisms
of $H_1(\Sigma)$ with a matrix of the following kind:
\[
\left(
\begin{array}{c|c}
\mathbbm{1} & 0 \\\hline
M &  \mathbbm{1}
\end{array}\right),
\]
where $M$ is an arbitrary $(n-1) \times 2g$ matrix.
Next we make use of the fact that for a surface $\Sigma'$ of genus $g$ with one boundary component,
the mapping class group surjects onto the symplectic group; see e.g.~\cite{BeMa},
where this is established for closed surfaces, which essentially implies the result for surfaces with one boundary component.
Since $\Sigma$ contains $\Sigma'$ as a subsurface, the following matrices may be realised as orientation-preserving diffeomorphisms:
\[
\left(
\begin{array}{c|c}
X & 0 \\\hline
0 & \mathbbm{1}
\end{array}\right),
\]
where $X$ is symplectic.
Finally, it is easy to see that boundary curves may be permuted (though by diffeomorphisms not coming from Dehn twists).
So, composing, one may realise any matrix of the form
\[
\left(
\begin{array}{c|c}
X & 0 \\\hline
M & P
\end{array}\right),
\]
where $X$ is symplectic, $M$ is arbitrary, and $P$ is a permutation matrix.
This completes the proof since such matrices are precisely those which represent an automorphism of
$H_1(\Sigma)$ that preserves the intersection form and permutes the homology classes of the boundary.
\end{proof}
A Seifert surface $\Sigma$ may inherit genus defect from an \emph{incompressible} subsurface,
i.e.~a subsurface $\Sigma'\subset\Sigma$ such that the induced map on the first homology group is injective.
\begin{lemma} \label{inherit}
Let $\Sigma$ be a Seifert surface of a link $L$, and let $\Sigma'\subset \Sigma$ be an incompressible subsurface
with boundary link $L'$.
If $L'$ bounds a slice surface $S'$, then $L$ bounds a slice
surface $S$ of genus $g(\Sigma) - g(\Sigma') + g(S')$.
In particular, if $g(\Sigma) = g(L)$, then $\Delta g (L) \geq \Delta g (L')$.
\end{lemma}
\begin{proof}
To construct $S$, simply cut out $\Sigma'$ and glue in $S'$.
\end{proof}
\begin{proof}[Proof of \cref{maintool}]
Let $B$ be a matrix of the Seifert form of $\Sigma$ restricted to $V$, with respect to an arbitrary basis of $V$.
Setting $t=1$ gives $\det(B - B^{\top}) = \pm 1$ and, in fact, $+1$:
indeed, $B-B^{\top}$ is antisymmetric, so $\det(B-B^{\top})$ is the square of its Pfaffian.
It also follows that $V$ is of even rank.
Because $B-B^{\top}$ is antisymmetric and unimodular,
we may assume the basis $x_{1,1}, \ldots, x_{k,1}, x_{1,2},\ldots, x_{k,2}$ of $V$ has been chosen
such that $B-B^{\top}$ is the $2k\times 2k$ matrix
\[
J_{k} =
\left(\!
\begin{array}{c|c}
0 & \mathbbm{1} \\\hline
- \mathbbm{1} & 0 \\
\end{array}\!\right).
\]
Let $\delta_1, \ldots, \delta_n$ be the 
boundary curves of $\Sigma$.
Note that the intersection form of $\Sigma$ is unimodular on $V$ (in fact it is represented by the matrix $B-B^{\top} = J_k$),
and identically zero on $\langle [\delta_1], \ldots, [\delta_{n-1}]\rangle$.
This implies that one can extend the basis of $V$ to a basis of $H_1(\Sigma)$ of the form
\begin{multline*}
x_{1,1}, \ldots, x_{k,1}, x_{1,2},\ldots, x_{k,2},\\
y_{1,1}, \ldots, y_{g-k,1},
        y_{1,2},\ldots, y_{g-k,2},
[\delta_1], \ldots, [\delta_{n-1}].
\end{multline*}
Let $A$ be the matrix of the Seifert form of $\Sigma$ with respect
to this basis. Then $A - A^{\top}$ has the form
\[
\left(
\begin{array}{c|c|c}
J_{k} & * & 0 \\\hline
* & * & 0  \\\hline
0 & 0 & 0  \\
\end{array}\right).
\]
Since $A-A^{\top}$ restricted to the span of the $x_{i,j}$ and $y_{i,j}$ is antisymmetric
and unimodular, one may assume w.l.o.g. that the $y_i$ were chosen such that $A-A^{\top}$ is in fact
\[
\left(
\begin{array}{c|c|c}
J_{k} & 0 & 0 \\\hline
0 & J_{g-k} & 0  \\\hline
0 & 0 & 0  \\
\end{array}\right).
\]
Now let
\[
\gamma_{1,1}, \gamma_{2,1}, \ldots, \gamma_{g,1}, \gamma_{1,2}, \ldots, \gamma_{g,2}, \delta_1, \ldots, \delta_{n-1}
\]
be a geometric basis on $\Sigma$ as in \cref{mpc}.
Let $\varphi$ be the automorphism of $H_1(\Sigma)$ given by
\begin{align*}
[\gamma_{i,j}] & \mapsto x_{i,j}   & \hspace{-7em}\text{for } 1 \leq i \leq k, \\
[\gamma_{i,j}] & \mapsto y_{i-k,j} & \hspace{-7em}\text{for } k < i \leq g, \\
[\delta_i]     & \mapsto [\delta_i].
\end{align*}
As the computation of $A-A^{\top}$ shows, $\varphi$ preserves the intersection form.
It also acts by permutation (in fact, as the identity) on the homology classes of boundary curves, and is therefore realised by a diffeomorphism $\widetilde{\varphi}$ (see \cref{mpc}).
Take a simple closed curve $\zeta$ that separates the curves $\gamma_{1,*},\ldots,\gamma_{k,*}$ from the curves $\gamma_{k+1,*},\ldots,\gamma_{g,*}, \delta_1, \ldots, \delta_n$.
Then $\widetilde{\varphi}(\zeta)$ is a separating simple closed curve, which bounds an
incompressible subsurface $\Sigma'$ of $\Sigma$ of genus $k$. By construction,
$H_1(\Sigma') = V \subset H_1(\Sigma)$, and hence the boundary knot $\widetilde{\varphi}(\zeta)$
of $\Sigma'$ has Alexander polynomial $1$. Thus,
Freedman's theorem  implies that it bounds a slice disc. Using \cref{inherit}, this concludes the proof.
\end{proof}
Let us come back to \cref{tildeX}. So far we have proved that $\partial\widetilde{X}$
has topological slice genus equal to one, while its classical genus equals
two. \Cref{plumbX} will show how this example can be used to build larger
examples with $\Delta g = g_4 = g/2$.
As a sample application, we calculate the topological slice genus of the infinite family provided in the proof of \cref{tree-limit}. These examples are of particular interest since they maximise the ratio
$$\frac{2\Delta g(L)}{b_1(L)}$$
of genus defect and first Betti number among tree-like plumbings of positive Hopf bands.
Indeed, for any plumbing of positive Hopf bands along a tree, this ratio is at most $1/3$ by a theorem of the fourth author~\cite{Lie}.
Therefore, an infinite family of examples that attain this ratio is sufficient to prove the following proposition.

\begin{prop}
\label{tree-limit}
 For the class of links arising as plumbings of positive Hopf bands along a finite tree, we have
 $$\limsup_{b_1(L) \to \infty} \frac{2\Delta g(L)}{b_1(L)} = \frac{1}{3}.$$
\end{prop}
\begin{lemma}\label{plumbX}
Let $\Sigma$ be a Seifert surface. Let $\Sigma'$ be a plumbing of $\Sigma$ and $\widetilde{X}$
along a square
on the right-most Hopf band of $\widetilde{X}$ (see \cref{figX}).
If there is an Alexander-trivial subgroup $V\subset H_1(\Sigma)$, then there is also
an Alexander-trivial subgroup $V' \subset H_1(\Sigma')$ of rank $\rk V' = 2 + \rk V$.
\end{lemma}
\begin{proof}
Let $\gamma_1, \gamma_2$ be the red and dashed blue curves on $\widetilde{X}$
as in \cref{tildeX}. Let $V' = V + \langle [\gamma_1], [\gamma_2]\rangle$, where
we understand $H_1(\Sigma) \oplus H_1(\widetilde{X})$ as a
subgroup of $H_1(\Sigma')$, because $\Sigma$ and $\widetilde{X}$ are incompressible
subsurfaces of $\Sigma'$, and $H_1(\Sigma) \cap H_1(\widetilde{X}) = \{0\}$.
The crucial observation is that, algebraically, $\gamma_1$ does not pass
through the plumbing location on the right-most Hopf band of $\widetilde{X}$.
Therefore, $\gamma_1$ algebraically does not intersect curves on $\Sigma$; and so, using that $\Sigma'$ is a plumbing, any small push-off of $\gamma_1$ along a normal direction of $\Sigma'$ has linking number 0 with curves on $\Sigma$.
Thus the Seifert form of $\Sigma'$ restricted to $V'$ is represented by the following matrix:
\[
    M' =
\left(\!
\begin{array}{ccc|cc}
  &   &  & 0 & * \\
  & M &  & \vdots & \vdots \\
  &   &  & 0 & * \\\hline
0 & \cdots & 0 & 0 & 0 \\
* & \cdots & * & 1 & 1 \\
\end{array}
\!\right).
\]
Here, $M$ is a matrix of the Seifert form restricted to $V$.
It has Alexander polynomial $1$, hence so does $M'$.
\end{proof}
\begin{figure}[ht]
 \includegraphics{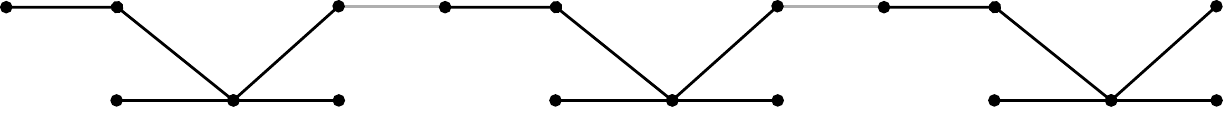}
\caption{How to stick together three copies of the tree corresponding to $\widetilde X$ in order to obtain a link with $b_1=18$.}
\label{k=3}
\end{figure}
\begin{proof}[Proof of \cref{tree-limit}]
In \cref{tildeX}, we considered such a link $L$ with $b_1=6$.
 In order to obtain an example $L_n$ with $b_1 = 6n$,
we simply stick together $n$ distinct copies of the tree corresponding to $\widetilde X$ and take the corresponding positive tree-like Hopf plumbing;
compare \cref{alotoftildeX} for an explicit braid description.
This is shown in \cref{k=3} for the case $n=3$.
By \cref{plumbX}, the corresponding fibre surface has defect $\Delta g\ge n$, which establishes the proposition.
\end{proof}
\begin{remark}
\label{alotoftildeX}
For all $n \geq 1$, the link $L_n$ used in the proof of \cref{tree-limit} can also be obtained as the closure of the $(3n+1)$-braid
\begin{multline*}
a_1 (a_1 a_3 a_2^2a_4a_1a_3a_2^2)(a_4a_6a_5^2a_7a_4a_6a_5^2)\cdots \\
(a_{3k-2}a_{3k}a_{3k-1}^2a_{3k+1}a_{3k-2}a_{3k}a_{3k-1}^2)\cdots(a_{3n-2}a_{3n}a_{3n-1}^2a_{3n-2}a_{3n}a_{3n-1}^2).
\end{multline*}
Furthermore, if we compare the topological slice genus with the classical genus (instead of the first Betti number), the quotient becomes even larger: since the links $L_n$ have topological slice genus $n$ and genus $2n$,
they form an infinite family of examples of positive braid links with $\Delta g = g_4 = g/2$.
\end{remark}
Next, let us focus on braids.
Incompressible subsurfaces of canonical Seifert surfaces of positive braids will typically
be constructed as in the following lemma, whose proof we leave to the reader.
We call a braid word $\beta'$ a \emph{subword} of a braid word $\beta$
if the former arises from the latter by deleting some occurrences of generators.
\begin{lemma}\label{subbraids}
If $\beta'$ is a subword of $\beta$, then
$\Sigma(\beta')$ is an incompressible subsurface of $\Sigma(\beta)$.\hfill$\Box$
\end{lemma}
\begin{figure}[ht]
\includegraphics[scale=0.8]{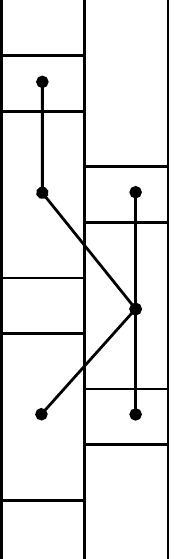}
\caption{A fence diagram (obtained from a braid diagram by replacing the crossings by horizontal line segments) of the braid $a_1(a_2^2a_1^2)^2$.
The tree induced by the distinguished homology generators exhibits $\widetilde X$ as incompressible subsurface of the fibre surface $\Sigma(a_1(a_2^2a_1^2)^2)$.}
\label{37torus}
\end{figure}
\begin{example}\label{t45}
Consider the torus knot $T(4,5)$. It is obtained as the closure of the positive braid $(a_1a_2a_3)^5$,
which contains the subword $a_1a_2^2a_3a_1a_2^3a_3$, whose closure equals the closure of $a_1a_3a_2^2a_1a_3a_2^3$.
In particular, the fibre surface $\Sigma(T(4,5))$ contains $\widetilde X$ as an incompressible subsurface.
Together with the bound coming from the signature function $\sigma_{e^{\pi i t}}(T(4,5))=10$ for $7/10 < t < 9/10$, this yields $g_4(T(4,5)) = 5$.
\end{example}
\begin{remark}\label{lt}
In \cref{t45}, we used half the absolute value of a Levine-Tristram signature as a lower bound for the topological slice genus of a knot. While well-known to experts, until recently this lower bound had not been explicitly stated in the literature in the topological setting (compare~\cite{tristram} for the smooth setting). This gap in the literature was closed by Powell with a new proof \cite{Powell_16}.
\end{remark}
\begin{example}\label{t37}
 Consider the torus knot $T(3,7)$. It is obtained as the closure of the positive braid $(a_1a_2)^7$,
 which contains $a_1(a_2^2a_1^2)^2$ as a subword.
 On the other hand, $\Sigma(a_1(a_2^2a_1^2)^2)$ contains $\widetilde X$ as an incompressible subsurface.
 This is schematically depicted in \cref{37torus}. Together with the bound coming from the signature function
 $\sigma_{e^{\pi i t}}(T(3,7))=10$ for $16/21 < t < 20/21$, this yields $g_4(T(3,7)) = 5$.
\end{example}
Suppose $\alpha$ and $\beta$ are braid words for non-split $n$-braids. Then $\Sigma(\alpha\beta)$ contains
$\Sigma(\alpha) \sqcup \Sigma(\beta)$ as incompressible subsurface. So if \cref{maintool}
produces genus defects $d_1$ and $d_2$ in $\Sigma(\alpha)$ and $\Sigma(\beta)$, respectively,
this will yield a defect of $d_1 + d_2$ in $\Sigma(\alpha\beta)$.
The following lemma is a refinement of this strategy for constructing genus defect in the product of two braids.
See \cref{kurzerwurm,langerwurm} for applications.
\begin{lemma}\label{zusammenstecken}
Let $\alpha, \beta$ be two braid words representing non-split $n$-braids.
Let $\beta'$ be the braid word of length $n-1$
obtained from $\beta$ by deleting for all $i\in\{1,\ldots, n-1\}$ all but the first
occurrences of the generator $a_i$.
Let $V \subset H_1(\Sigma({\alpha\beta'}))$ and
$V' \subset H_1(\Sigma({\beta}))$ be
Alexander-trivial subgroups.
Let a basis of $V$ be given with respect to which the Seifert form of $\Sigma(\alpha\beta')$ has a matrix
of the following kind, built from four square blocks:%
\footnote{For example, all \emph{trivial Alexander bases} \cite{TeGa} are of this kind (but not vice versa).
In fact, one can prove that every Alexander-trivial subgroup $V$ has such a basis; but we will not need this fact here.}
\[
\newcommand{\bigzero}{0}
\newcommand{\bigast}{*}
\left(\!
\begin{array}{c|c}
\bigzero &
\begin{array}{*{3}{c@{\ }}} 1 && \bigast \\[-2ex] & \ddots & \\[-1ex] \bigzero && 1 \end{array}
\\\hline
\begin{array}{*{3}{c@{\ }}} 0 && 0 \\[-2ex] & \ddots & \\[-1ex] \bigast && 0 \end{array}
& \bigast
\end{array}\!\right).
\]
Suppose moreover that the first half of that basis
is supported in $H_1(\Sigma({\alpha}))$, which can be seen as a subgroup of $H_1(\Sigma({\alpha\beta'}))$
by \cref{subbraids}.
Then there is an Alexander-trivial subgroup $V'' \subset H_1(\Sigma(\alpha\beta))$ of rank $\rk V + \rk V'$.
\end{lemma}
\begin{proof}
The idea is similar to the proof of \cref{plumbX}. The surface $\Sigma(\alpha\beta)$ has incompressible
subsurfaces $\Sigma(\alpha\beta')$ and $\Sigma(\beta)$, so we may treat
$H_1(\Sigma(\alpha\beta'))$ and $H_1(\Sigma(\beta))$ as subgroups of $H_1(\Sigma(\alpha\beta))$.
Their intersection is in fact trivial, and so we have $V\cap V' = \{0\}$ as well.
Extend the given basis of $V$ to a basis of $V + V'$. With respect to this basis,
the restriction of the Seifert form of $\Sigma(\alpha\beta)$ to $V + V'$ is represented by the following matrix:
\[
\newcommand{\bigzero}{0}
\newcommand{\bigast}{*}
M' =
\left(\!
\begin{array}{c|c|c}
\bigzero &
\begin{array}{*{3}{c@{\ }}} 1 && \bigast \\[-2ex] & \ddots & \\[-1ex] \bigzero && 1 \end{array} & 0
\\\hline
\begin{array}{*{3}{c@{\ }}} 0 && 0 \\[-2ex] & \ddots & \\[-1ex] \bigast && 0 \end{array}
& \bigast & \bigast \\\hline
0 & \bigast & M
\end{array}\!\right).
\]
Here, $M$ is the matrix of the Seifert form restricted to $V'$, which has Alexander polynomial $1$.
After some basis changes, one sees that $M'$ has Alexander polynomial $1$ as well.
\end{proof}
\begin{example}\label{kurzerwurm}
We have seen in \cref{t37} how the closure of
\[
\beta = a_1(a_2^2a_1^2)^2
\]
has defect at least one,
which comes from two vectors $v, w$ restricted to which the Seifert form has
the matrix
\[
\begin{pmatrix}
0 & 1 \\
0 & *
\end{pmatrix}.
\]
The vectors $v$ and $w$ are the homology classes of the red and blue curves
drawn in \cref{figX}. As already discussed in the proof of \cref{plumbX}, $v \in H_1(\Sigma(\alpha)) \subset H_1(\Sigma(\beta))$,
where $\alpha = a_1a_2^2a_1^2a_2^2a_1$. Let $\beta' = a_1a_2$ as in the previous lemma.
Then $\alpha\beta'$ contains $\beta$ as a subword, and so $\Sigma(\alpha\beta')$ also has defect at least $1$.
So the previous lemma implies that $\alpha\beta = a_1(a_2^2a_1^2)^4$ has defect
at least $2$. Continuing inductively, one finds a defect of at least $i$ in the
closure of the braid
\[
\alpha^{i-1}\beta = a_1(a_2^2a_1^2)^{2i}.
\]
The same result may be obtained using \cref{plumbX}, since $\widetilde{X}\subset \Sigma(\beta)$,
as shown in \cref{37torus}.
\end{example}
\begin{remark}\label{comp}
\Cref{maintool} shows how to construct slice surfaces using nothing but linear algebra.
The following randomised algorithm exploits this. As input, it takes
an arbitrary integral square matrix $A$, and returns as output the basis of a subgroup $V\subset \mathbb{Z}^{2g}$
with respect to which $A|_V$ has a matrix of the following kind:
\[
\newcommand{\bigzero}{0}
\newcommand{\bigast}{*}
\left(\!
\begin{array}{c|c}
\bigzero &
\begin{array}{*{3}{c@{\ }}} 1 && \bigzero \\[-2ex] & \ddots & \\[-1ex] \bigzero && 1 \end{array}
\\\hline
\begin{array}{*{3}{c@{\ }}} 0 && \bigast \\[-2ex] & \ddots & \\[-1ex] \bigzero && 0 \end{array}
& \bigast
\end{array}\!\right).
\]
Note that such a matrix has Alexander polynomial 1.
Here is a brief description of the algorithm:
\begin{enumerate}
\item Randomly pick a primitive vector $v$ with $v^{\top}Av = 0$, if such a vector exists.
Otherwise, return the empty basis.
\item Randomly pick a solution $w$ of the following system of linear equations, if it is solvable:
\[
v^{\top}Aw  = 1, \qquad
w^{\top}Av  = 0.
\]
Otherwise, go back to (1), or eventually give up and return the empty basis.
\item Let $U$ be the subgroup of solutions of the following system of homogeneous linear equations:
\[
v^{\top}Au  = 0, \qquad
u^{\top}Av  = 0, \qquad
u^{\top}Aw  = 0.
\]
Let $(v_1, \ldots, v_k, w_1, \ldots, w_k)$ be the result of the recursive application of the algorithm to $A|_U$.
Return $$(v, v_1,\ldots, v_k, w, w_1, \ldots, w_k).$$
\end{enumerate}
Implemented in pari/gp \cite{pari}, the algorithm performs quite well for small knots.
See \cref{table64} for the results thus obtained for small torus knots,
and \cref{langerwurm} for the application to another positive braid.
The bases of the respective subgroups $V$ are available from ancillary files with the arXiv-version of this paper,
which enables anybody to independently verify their correctness.
\end{remark}
\begin{example} \label{langerwurm}
\newcommand{\olomega}{\widetilde{\omega}}
Consider the positive braids $\omega  = a_1a_2a_3a_4, \olomega  = a_4a_3a_2a_1$.
The algorithm described in~\cref{comp} returns an Alexander-trivial subgroup $V\subset H_1(\Sigma((\omega\olomega)^4))$ of rank eight
(we used \cite{collins} to obtain Seifert matrices).
Moreover, the first half of the basis of $V$ is supported in $H_1(\Sigma((\omega\olomega)^3\omega))$.
Similarly, there is an Alexander-trivial
subgroup $V' \subset H_1(\Sigma((\olomega\omega)^4))$ of rank eight with a basis whose first half is supported in $H_1(\Sigma((\olomega\omega)^3\olomega))$. %
Applying~\cref{zusammenstecken} to $(\omega\olomega)^3\omega$ and $(\olomega\omega)^4$
gives a defect of eight in $\Sigma((\omega\olomega)^7\omega)$.
We may continue applying the lemma inductively, first to $(\omega\olomega)^7$ and $(\omega\olomega)^4$,
producing a defect of twelve in $\Sigma((\omega\olomega)^{11})$, then to
$(\omega\olomega)^{10}\omega$ and $(\olomega\omega)^4$ etc.
In summary, we find for all $i\geq 0$ a defect of $4+8i$ for $\Sigma((\omega\olomega)^{4+7i})$,
and of $8i$ for $\Sigma((\omega\olomega)^{7i}\omega)$.
\end{example}

\section{Slice genus of large torus knots}
\label{sec3}
The aim of this section is to prove the asymptotic bound for the genus defect of torus knots given by \cref{T2}.
As a start, we establish a weaker version of \cref{T2} with the benefit that its proof, unlike the proof of \cref{T2}, does not require computer calculations.
The strategies of both proofs are very much alike.

\begin{prop}
\label{fifth}
$$\lim_{n,m\to\infty}\frac{g_4(T(n,m))}{g(T(n,m))}\le\frac{4}{5}.$$
\end{prop}

The strategy of the proof of \cref{fifth} is to establish that the fibre surface $\Sigma(T(n,n))$ of the torus link $T(n,n)$ contains
as incompressible subsurface
the split union of fibre surfaces of the form $\Sigma(a_1(a_1^2a_2^2)^{2i})$ such that this union takes up roughly four-fifths of the genus of $\Sigma(T(n,n))$.
This yields \cref{fifth} since the genus defect of the closure of ${a_1(a_1^2a_2^2)^{2i}}$ is at least $i$, which is about a quarter of the genus. Indeed, first conjugating by $a_1$ and then reading the braid word backwards (both of these operations preserve the closure up to changing the orientation of all components) turns ${a_1(a_1^2a_2^2)^{2i}}$ into ${a_1(a_2^2a_1^2)^{2i}}$, whose defect is discussed in \cref{kurzerwurm}.
To make this strategy precise we use \cref{subsurface}.
Let $\Delta_n$ be the half twist on $n$ strands, i.e. $$\Delta_n = (a_1a_2\cdots a_{n-1})(a_1a_2\cdots a_{n-2})\cdots (a_1a_2)(a_1).$$
Furthermore, we define the positive braids $\Omega_i$ and $\Gamma_j$ by
\begin{align*}
\Omega_{i} &= a_1a_2\cdots a_{i-2}a_{i-1}^2a_{i-2}\cdots a_2a_1, \\
\Gamma_j &= a_1a_2\cdots a_{j-2}a_{j-1}a_{j-2}\cdots a_2a_1.
\end{align*}

\begin{figure}[ht]
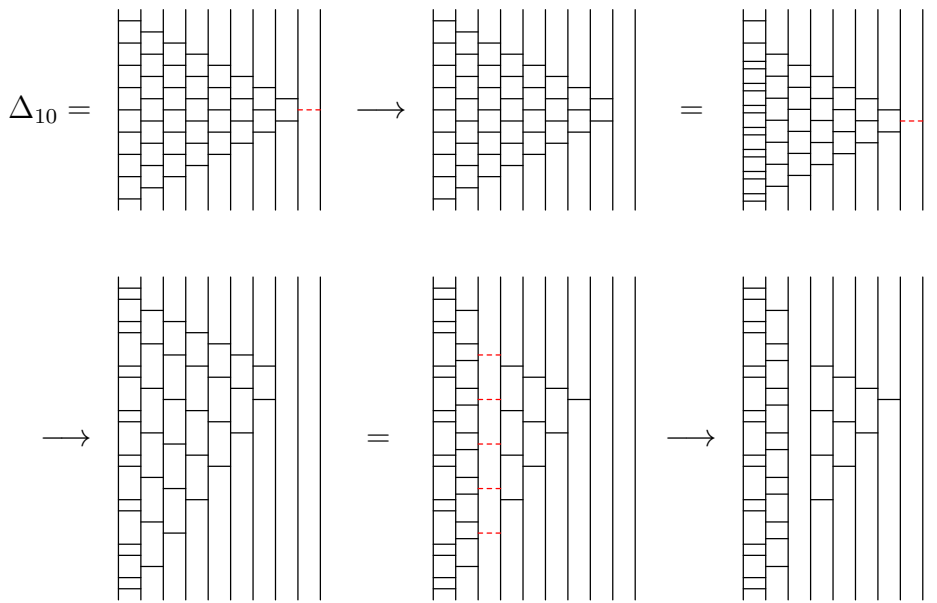

\centering
\begin{tabular}{rlll}
$\Delta_{10}=$ & $\xygraph{
!{0;/r0.7pc/:}
[u(4.5)]!{\xcapv[9]@(0)}
[lu]!{\xcapv[9]@(0)}
[lu]!{\xcapv[9]@(0)}
[lu]!{\xcapv[9]@(0)}
[lu]!{\xcapv[9]@(0)}
[lu]!{\xcapv[9]@(0)}
[lu]!{\xcapv[9]@(0)}
[lu]!{\xcapv[9]@(0)}
[lu]!{\xcapv[9]@(0)}
[lu]!{\xcapv[9]@(0)}
[u(0.5)] !{\xcaph[1]@(0)}
[d(0.5)] !{\xcaph[1]@(0)}
[d(0.5)] !{\xcaph[1]@(0)}
[d(0.5)] !{\xcaph[1]@(0)}
[d(0.5)] !{\xcaph[1]@(0)}
[d(0.5)] !{\xcaph[1]@(0)}
[d(0.5)] !{\xcaph[1]@(0)}
[d(0.5)] !{\xcaph[1]@(0)}
[d(0.5)] !{\color{red}\xcaph[0.2]@(0)} [l(0.6)] !{\xcaph[0.2]@(0)} [l(0.6)] !{\xcaph[0.2]@(0)} [l(0.8)]
[u(3)lllllllll] !{\color{black}\xcaph[1]@(0)}
[d(0.5)] !{\xcaph[1]@(0)}
[d(0.5)] !{\xcaph[1]@(0)}
[d(0.5)] !{\xcaph[1]@(0)}
[d(0.5)] !{\xcaph[1]@(0)}
[d(0.5)] !{\xcaph[1]@(0)}
[d(0.5)] !{\xcaph[1]@(0)}
[d(0.5)] !{\xcaph[1]@(0)}
[u(2.5)llllllll] !{\xcaph[1]@(0)}
[d(0.5)] !{\xcaph[1]@(0)}
[d(0.5)] !{\xcaph[1]@(0)}
[d(0.5)] !{\xcaph[1]@(0)}
[d(0.5)] !{\xcaph[1]@(0)}
[d(0.5)] !{\xcaph[1]@(0)}
[d(0.5)] !{\xcaph[1]@(0)}
[u(2)lllllll] !{\xcaph[1]@(0)}
[d(0.5)] !{\xcaph[1]@(0)}
[d(0.5)] !{\xcaph[1]@(0)}
[d(0.5)] !{\xcaph[1]@(0)}
[d(0.5)] !{\xcaph[1]@(0)}
[d(0.5)] !{\xcaph[1]@(0)}
[u(1.5)llllll] !{\xcaph[1]@(0)}
[d(0.5)] !{\xcaph[1]@(0)}
[d(0.5)] !{\xcaph[1]@(0)}
[d(0.5)] !{\xcaph[1]@(0)}
[d(0.5)] !{\xcaph[1]@(0)}
[u(1)lllll] !{\xcaph[1]@(0)}
[d(0.5)] !{\xcaph[1]@(0)}
[d(0.5)] !{\xcaph[1]@(0)}
[d(0.5)] !{\xcaph[1]@(0)}
[u(0.5)llll] !{\xcaph[1]@(0)}
[d(0.5)] !{\xcaph[1]@(0)}
[d(0.5)] !{\xcaph[1]@(0)}
[lll] !{\xcaph[1]@(0)}
[d(0.5)] !{\xcaph[1]@(0)}
[d(0.5)ll] !{\xcaph[1]@(0)}
} \hspace{-5.5pc}
\longrightarrow$
& $\xygraph{
!{0;/r0.7pc/:}
[u(4.5)]!{\xcapv[9]@(0)}
[lu]!{\xcapv[9]@(0)}
[lu]!{\xcapv[9]@(0)}
[lu]!{\xcapv[9]@(0)}
[lu]!{\xcapv[9]@(0)}
[lu]!{\xcapv[9]@(0)}
[lu]!{\xcapv[9]@(0)}
[lu]!{\xcapv[9]@(0)}
[lu]!{\xcapv[9]@(0)}
[lu]!{\xcapv[9]@(0)}
[u(0.5)] !{\xcaph[1]@(0)}
[d(0.5)] !{\xcaph[1]@(0)}
[d(0.5)] !{\xcaph[1]@(0)}
[d(0.5)] !{\xcaph[1]@(0)}
[d(0.5)] !{\xcaph[1]@(0)}
[d(0.5)] !{\xcaph[1]@(0)}
[d(0.5)] !{\xcaph[1]@(0)}
[d(0.5)] !{\xcaph[1]@(0)}
[u(2.5)llllllll] !{\xcaph[1]@(0)}
[d(0.5)] !{\xcaph[1]@(0)}
[d(0.5)] !{\xcaph[1]@(0)}
[d(0.5)] !{\xcaph[1]@(0)}
[d(0.5)] !{\xcaph[1]@(0)}
[d(0.5)] !{\xcaph[1]@(0)}
[d(0.5)] !{\xcaph[1]@(0)}
[d(0.5)] !{\xcaph[1]@(0)}
[u(2.5)llllllll] !{\xcaph[1]@(0)}
[d(0.5)] !{\xcaph[1]@(0)}
[d(0.5)] !{\xcaph[1]@(0)}
[d(0.5)] !{\xcaph[1]@(0)}
[d(0.5)] !{\xcaph[1]@(0)}
[d(0.5)] !{\xcaph[1]@(0)}
[d(0.5)] !{\xcaph[1]@(0)}
[u(2)lllllll] !{\xcaph[1]@(0)}
[d(0.5)] !{\xcaph[1]@(0)}
[d(0.5)] !{\xcaph[1]@(0)}
[d(0.5)] !{\xcaph[1]@(0)}
[d(0.5)] !{\xcaph[1]@(0)}
[d(0.5)] !{\xcaph[1]@(0)}
[u(1.5)llllll] !{\xcaph[1]@(0)}
[d(0.5)] !{\xcaph[1]@(0)}
[d(0.5)] !{\xcaph[1]@(0)}
[d(0.5)] !{\xcaph[1]@(0)}
[d(0.5)] !{\xcaph[1]@(0)}
[u(1)lllll] !{\xcaph[1]@(0)}
[d(0.5)] !{\xcaph[1]@(0)}
[d(0.5)] !{\xcaph[1]@(0)}
[d(0.5)] !{\xcaph[1]@(0)}
[u(0.5)llll] !{\xcaph[1]@(0)}
[d(0.5)] !{\xcaph[1]@(0)}
[d(0.5)] !{\xcaph[1]@(0)}
[lll] !{\xcaph[1]@(0)}
[d(0.5)] !{\xcaph[1]@(0)}
[d(0.5)ll] !{\xcaph[1]@(0)}
} \hspace{-5.2pc}
=$
& $\xygraph{
!{0;/r0.7pc/:}
[u(4.5)]!{\xcapv[9]@(0)}
[lu]!{\xcapv[9]@(0)}
[lu]!{\xcapv[9]@(0)}
[lu]!{\xcapv[9]@(0)}
[lu]!{\xcapv[9]@(0)}
[lu]!{\xcapv[9]@(0)}
[lu]!{\xcapv[9]@(0)}
[lu]!{\xcapv[9]@(0)}
[lu]!{\xcapv[9]@(0)}
[lu]!{\xcapv[9]@(0)}
[u(0.5)] !{\xcaph[1]@(0)}
[d(1)l] !{\xcaph[1]@(0)}
[d(0.5)] !{\xcaph[1]@(0)}
[d(0.5)] !{\xcaph[1]@(0)}
[d(0.5)] !{\xcaph[1]@(0)}
[d(0.5)] !{\xcaph[1]@(0)}
[d(0.5)] !{\xcaph[1]@(0)}
[d(0.5)] !{\xcaph[1]@(0)}
[d(0.5)] !{\color{red}\xcaph[0.2]@(0)} [l(0.6)] !{\xcaph[0.2]@(0)} [l(0.6)] !{\xcaph[0.2]@(0)} [l(0.8)]
[u(2.677)llllllll] !{\color{black}
\xcaph[1]@(0)}
[d(0.333)l] !{\xcaph[1]@(0)}
[d(0.333)] !{\xcaph[1]@(0)}
[d(0.5)] !{\xcaph[1]@(0)}
[d(0.5)] !{\xcaph[1]@(0)}
[d(0.5)] !{\xcaph[1]@(0)}
[d(0.5)] !{\xcaph[1]@(0)}
[d(0.5)] !{\xcaph[1]@(0)}
[u(2.177)lllllll] !{\xcaph[1]@(0)}
[d(0.333)l] !{\xcaph[1]@(0)}
[d(0.333)] !{\xcaph[1]@(0)}
[d(0.5)] !{\xcaph[1]@(0)}
[d(0.5)] !{\xcaph[1]@(0)}
[d(0.5)] !{\xcaph[1]@(0)}
[d(0.5)] !{\xcaph[1]@(0)}
[u(1.677)llllll] !{\xcaph[1]@(0)}
[d(0.333)l] !{\xcaph[1]@(0)}
[d(0.333)] !{\xcaph[1]@(0)}
[d(0.5)] !{\xcaph[1]@(0)}
[d(0.5)] !{\xcaph[1]@(0)}
[d(0.5)] !{\xcaph[1]@(0)}
[u(1.177)lllll] !{\xcaph[1]@(0)}
[d(0.333)l] !{\xcaph[1]@(0)}
[d(0.333)] !{\xcaph[1]@(0)}
[d(0.5)] !{\xcaph[1]@(0)}
[d(0.5)] !{\xcaph[1]@(0)}
[u(0.677)llll] !{\xcaph[1]@(0)}
[d(0.333)l] !{\xcaph[1]@(0)}
[d(0.333)] !{\xcaph[1]@(0)}
[d(0.5)] !{\xcaph[1]@(0)}
[u(0.177)lll] !{\xcaph[1]@(0)}
[d(0.333)l] !{\xcaph[1]@(0)}
[d(0.333)] !{\xcaph[1]@(0)}
[d(0.333)ll] !{\xcaph[1]@(0)}
[d(0.333)l] !{\xcaph[1]@(0)}
}$ \hspace{-130pt}
\\[10ex]
$\longrightarrow$
& $\xygraph{
!{0;/r0.7pc/:}
[u(7.25)]!{\xcapv[14.5]@(0)}
[lu]!{\xcapv[14.5]@(0)}
[lu]!{\xcapv[14.5]@(0)}
[lu]!{\xcapv[14.5]@(0)}
[lu]!{\xcapv[14.5]@(0)}
[lu]!{\xcapv[14.5]@(0)}
[lu]!{\xcapv[14.5]@(0)}
[lu]!{\xcapv[14.5]@(0)}
[lu]!{\xcapv[14.5]@(0)}
[lu]!{\xcapv[14.5]@(0)}
[u(0.5)] !{\xcaph[1]@(0)}
[d(0.5)l] !{\xcaph[1]@(0)}
[d(0.5)] !{\xcaph[1]@(0)}
[d(0.5)] !{\xcaph[1]@(0)}
[d(0.5)] !{\xcaph[1]@(0)}
[d(0.5)] !{\xcaph[1]@(0)}
[d(0.5)] !{\xcaph[1]@(0)}
[d(0.5)] !{\xcaph[1]@(0)}
[u(2)lllllll] !{\xcaph[1]@(0)}
[d(0.5)l] !{\xcaph[1]@(0)}
[d(0.5)] !{\xcaph[1]@(0)}
[d(0.5)] !{\xcaph[1]@(0)}
[d(0.5)] !{\xcaph[1]@(0)}
[d(0.5)] !{\xcaph[1]@(0)}
[d(0.5)] !{\xcaph[1]@(0)}
[d(0.5)] !{\xcaph[1]@(0)}
[u(1.5)lllllll] !{\xcaph[1]@(0)}
[d(0.5)l] !{\xcaph[1]@(0)}
[d(0.5)] !{\xcaph[1]@(0)}
[d(0.5)] !{\xcaph[1]@(0)}
[d(0.5)] !{\xcaph[1]@(0)}
[d(0.5)] !{\xcaph[1]@(0)}
[d(0.5)] !{\xcaph[1]@(0)}
[u(1)llllll] !{\xcaph[1]@(0)}
[d(0.5)l] !{\xcaph[1]@(0)}
[d(0.5)] !{\xcaph[1]@(0)}
[d(0.5)] !{\xcaph[1]@(0)}
[d(0.5)] !{\xcaph[1]@(0)}
[d(0.5)] !{\xcaph[1]@(0)}
[u(0.5)lllll] !{\xcaph[1]@(0)}
[d(0.5)l] !{\xcaph[1]@(0)}
[d(0.5)] !{\xcaph[1]@(0)}
[d(0.5)] !{\xcaph[1]@(0)}
[d(0.5)] !{\xcaph[1]@(0)}
[llll] !{\xcaph[1]@(0)}
[d(0.5)l] !{\xcaph[1]@(0)}
[d(0.5)] !{\xcaph[1]@(0)}
[d(0.5)] !{\xcaph[1]@(0)}
[d(0.5)lll] !{\xcaph[1]@(0)}
[d(0.5)l] !{\xcaph[1]@(0)}
[d(0.5)] !{\xcaph[1]@(0)}
[d(0.5)ll] !{\xcaph[1]@(0)}
[d(0.5)l] !{\xcaph[1]@(0)}
} \hspace{-9pc}
=$
& $\xygraph{
!{0;/r0.7pc/:}
[u(7.25)]!{\xcapv[14.5]@(0)}
[lu]!{\xcapv[14.5]@(0)}
[lu]!{\xcapv[14.5]@(0)}
[lu]!{\xcapv[14.5]@(0)}
[lu]!{\xcapv[14.5]@(0)}
[lu]!{\xcapv[14.5]@(0)}
[lu]!{\xcapv[14.5]@(0)}
[lu]!{\xcapv[14.5]@(0)}
[lu]!{\xcapv[14.5]@(0)}
[lu]!{\xcapv[14.5]@(0)}
[u(0.5)] !{\xcaph[1]@(0)}
[d(0.5)l] !{\xcaph[1]@(0)}
[d(0.5)] !{\xcaph[1]@(0)}
[d(0.5)ll] !{\xcaph[1]@(0)}
[d(0.5)l] !{\xcaph[1]@(0)}
[d(0.5)] !{\xcaph[1]@(0)}
[d(0.5)] !{\color{red}\xcaph[0.2]@(0)} [l(0.6)] !{\xcaph[0.2]@(0)} [l(0.6)] !{\xcaph[0.2]@(0)\color{black}} [l(0.8)]
[d(0.5)] !{\xcaph[1]@(0)}
[d(0.5)] !{\xcaph[1]@(0)}
[d(0.5)] !{\xcaph[1]@(0)}
[d(0.5)] !{\xcaph[1]@(0)}
[u(1.75)llllll] !{\color{black}\xcaph[1]@(0)}
[d(0.25)ll] !{\xcaph[1]@(0)}
[d(0.5)l] !{\xcaph[1]@(0)}
[d(0.5)] !{\xcaph[1]@(0)}
[d(0.5)] !{\color{red}\xcaph[0.2]@(0)} [l(0.6)] !{\xcaph[0.2]@(0)} [l(0.6)] !{\xcaph[0.2]@(0)\color{black}} [l(0.8)]
[d(0.5)] !{\xcaph[1]@(0)}
[d(0.5)] !{\xcaph[1]@(0)}
[d(0.5)] !{\xcaph[1]@(0)}
[u(1.25)lllll] !{\color{black}\xcaph[1]@(0)}
[d(0.25)ll] !{\xcaph[1]@(0)}
[d(0.5)l] !{\xcaph[1]@(0)}
[d(0.5)] !{\xcaph[1]@(0)}
[d(0.5)]  !{\color{red}\xcaph[0.2]@(0)} [l(0.6)] !{\xcaph[0.2]@(0)} [l(0.6)] !{\xcaph[0.2]@(0)\color{black}} [l(0.8)]
[d(0.5)] !{\xcaph[1]@(0)}
[d(0.5)] !{\xcaph[1]@(0)}
[u(0.75)llll] !{\color{black}\xcaph[1]@(0)}
[d(0.25)ll] !{\xcaph[1]@(0)}
[d(0.5)l] !{\xcaph[1]@(0)}
[d(0.5)] !{\xcaph[1]@(0)}
[d(0.5)]  !{\color{red}\xcaph[0.2]@(0)} [l(0.6)] !{\xcaph[0.2]@(0)} [l(0.6)] !{\xcaph[0.2]@(0)\color{black}} [l(0.8)]
[d(0.5)] !{\xcaph[1]@(0)}
[u(0.25)lll] !{\color{black}\xcaph[1]@(0)}
[d(0.25)ll] !{\xcaph[1]@(0)}
[d(0.5)l] !{\xcaph[1]@(0)}
[d(0.5)] !{\xcaph[1]@(0)}
[d(0.5)]  !{\color{red}\xcaph[0.2]@(0)} [l(0.6)] !{\xcaph[0.2]@(0)} [l(0.6)] !{\xcaph[0.2]@(0)\color{black}} [l(0.8)]
[d(0.25)ll] !{\xcaph[1]@(0)}
[d(0.25)ll] !{\xcaph[1]@(0)}
[d(0.5)l] !{\xcaph[1]@(0)}
[d(0.5)] !{\xcaph[1]@(0)}
[d(0.5)ll] !{\xcaph[1]@(0)}
[d(0.5)l] !{\xcaph[1]@(0)}
} \hspace{-9.5pc} \longrightarrow$
& $\xygraph{
!{0;/r0.7pc/:}
[u(7.25)]!{\xcapv[14.5]@(0)}
[lu]!{\xcapv[14.5]@(0)}
[lu]!{\xcapv[14.5]@(0)}
[lu]!{\xcapv[14.5]@(0)}
[lu]!{\xcapv[14.5]@(0)}
[lu]!{\xcapv[14.5]@(0)}
[lu]!{\xcapv[14.5]@(0)}
[lu]!{\xcapv[14.5]@(0)}
[lu]!{\xcapv[14.5]@(0)}
[lu]!{\xcapv[14.5]@(0)}
[u(0.5)] !{\xcaph[1]@(0)}
[d(0.5)l] !{\xcaph[1]@(0)}
[d(0.5)] !{\xcaph[1]@(0)}
[d(0.5)ll] !{\xcaph[1]@(0)}
[d(0.5)l] !{\xcaph[1]@(0)}
[d(0.5)] !{\xcaph[1]@(0)}
[d(1.0)r] !{\xcaph[1]@(0)}
[d(0.5)] !{\xcaph[1]@(0)}
[d(0.5)] !{\xcaph[1]@(0)}
[d(0.5)] !{\xcaph[1]@(0)}
[u(1.75)llllll] !{\xcaph[1]@(0)}
[d(0.25)ll] !{\xcaph[1]@(0)}
[d(0.5)l] !{\xcaph[1]@(0)}
[d(0.5)] !{\xcaph[1]@(0)}
[d(1.0)r] !{\xcaph[1]@(0)}
[d(0.5)] !{\xcaph[1]@(0)}
[d(0.5)] !{\xcaph[1]@(0)}
[u(1.25)lllll] !{\color{black}\xcaph[1]@(0)}
[d(0.25)ll] !{\xcaph[1]@(0)}
[d(0.5)l] !{\xcaph[1]@(0)}
[d(0.5)] !{\xcaph[1]@(0)}
[d(1.0)r] !{\xcaph[1]@(0)}
[d(0.5)] !{\xcaph[1]@(0)}
[u(0.75)llll] !{\color{black}\xcaph[1]@(0)}
[d(0.25)ll] !{\xcaph[1]@(0)}
[d(0.5)l] !{\xcaph[1]@(0)}
[d(0.5)] !{\xcaph[1]@(0)}
[d(1.0)r] !{\xcaph[1]@(0)}
[u(0.25)lll] !{\color{black}\xcaph[1]@(0)}
[d(0.25)ll] !{\xcaph[1]@(0)}
[d(0.5)l] !{\xcaph[1]@(0)}
[d(0.5)] !{\xcaph[1]@(0)}
[d(0.75)l] !{\xcaph[1]@(0)}
[d(0.25)ll] !{\xcaph[1]@(0)}
[d(0.5)l] !{\xcaph[1]@(0)}
[d(0.5)] !{\xcaph[1]@(0)}
[d(0.5)ll] !{\xcaph[1]@(0)}
[d(0.5)l] !{\xcaph[1]@(0)}
}$ \hspace{-130pt}
\end{tabular}
\caption{By applying braid relations and deleting generators, the $10$-strand braid word $\Delta_{10}$ is transformed into
$\Gamma_2\Gamma_3\Gamma_4a_4a_5a_6a_7\Gamma_4a_4a_5a_6\Gamma_4a_4a_5\Gamma_4a_4\Gamma_4\Gamma_3\Gamma_2$.
In the final step, deleting generators produces the disjoint union of $\Gamma_2\Gamma_3 \Omega_3^5 \Gamma_3\Gamma_2$ and $\Delta_5$.
Arrows indicate the deletion of generators drawn red and dashed.%
}
\label{fig:Deltatoaabb}
\end{figure}

\begin{lemma}
\label{subsurface}
Let $n\ge 2\ell$ be natural numbers. Then $\Sigma(\Delta_{n})$ contains $$\Sigma(\Gamma_2\cdots\Gamma_{\ell}\Omega_\ell^{n-2\ell+1}\Gamma_\ell\cdots\Gamma_2) \sqcup \Sigma(\Delta_{n-2\ell+1})$$ as an incompressible subsurface.
\end{lemma}

\begin{proof}
We proceed by showing that one can delete generators and apply braid relations in the braid word
$\Delta_{n}$ such that the resulting positive braid is the split union of the positive braids $\Gamma_2\cdots\Gamma_{\ell}\Omega_\ell^{n-2\ell+1}\Gamma_\ell\cdots\Gamma_2$ and $\Delta_{n-2\ell+1}$.
This suffices to establish \cref{subsurface} since deleting a generator in a positive braid word corresponds to taking an incompressible subsurface of the associated fibre surface (see \cref{subbraids}).
We start by considering the positive braid word
$$\Delta_{n}= \Gamma_2(a_2\cdots a_{n-1})\Gamma_2(a_2\cdots a_{n-2})\cdots \Gamma_2 (a_2a_3)\Gamma_2 (a_2)\Gamma_2.$$
We delete the single occurrence of the generator $a_{n-1}$ in $\Delta_{n}$
and then apply braid relations to obtain the positive braid word
$$\Gamma_2\Gamma_3(a_3\cdots a_{n-2})\Gamma_3(a_3\cdots a_{n-3})\cdots \Gamma_3 (a_3a_4)\Gamma_3 (a_3) \Gamma_3\Gamma_2.$$
This can be achieved by multiple substitutions of the form 
$$(a_i\cdots a_j)\Gamma_i(a_i\cdots a_{j})\to\Gamma_{i+1}(a_{i+1}\cdots a_j)(a_i\cdots a_{j-1})$$
for $i \leq j$, which can in turn be realised by braid relations. To see the realisation of this substitution by braid relations, commute generators to rewrite the positive braid word
$$(a_i\cdots a_j)\Gamma_i(a_i\cdots a_{j})$$
as
$$a_i\Gamma_ia_{i+1}a_{i}a_{i+2}a_{i+1}\cdots a_{j-1}a_{j-2}a_ja_{j-1}a_j.$$
Then, applying the braid relation 
$$a_ka_{k-1}a_k \to a_{k-1}a_ka_{k-1}$$
once for each $k$ starting at $j$ and descending down to $i$ yields the positive braid word
$$\Gamma_{i+1}a_{i+1}a_{i}a_{i+2}\cdots a_{j-3}a_{j-1}a_{j-2}a_ja_{j-1},$$
for which generators can again be commuted to finally result in
$$\Gamma_{i+1}(a_{i+1}\cdots a_j)(a_i\cdots a_{j-1}).$$

In the next step, we delete the single occurrence of the generator $a_{n-2}$ in the positive braid word 
$$\Gamma_2\Gamma_3(a_3\cdots a_{n-2})\Gamma_3(a_3\cdots a_{n-3})\cdots \Gamma_3 (a_3a_4)\Gamma_3 (a_3) \Gamma_3\Gamma_2$$
and, again
using substitutions of the form 
$$(a_i\cdots a_j)\Gamma_i(a_i\cdots a_{j})\to\Gamma_{i+1}(a_{i+1}\cdots a_j)(a_i\cdots a_{j-1}),$$
obtain the positive braid word
$$\Gamma_2\Gamma_3\Gamma_4(a_4\cdots a_{n-3})\Gamma_4(a_4\cdots a_{n-4})\cdots \Gamma_4 (a_4a_5)\Gamma_4 (a_4) \Gamma_4\Gamma_3\Gamma_2.$$
We continue in the same way until we arrive at the positive braid word
\begin{multline*}
\Gamma_2\cdots\Gamma_{\ell+1}(a_{\ell+1}\cdots a_{n-\ell})\Gamma_{\ell+1}(a_{\ell+1}\cdots a_{n-\ell-1})\cdots\\
        \Gamma_{\ell+1} (a_{\ell+1}a_{\ell+2})\Gamma_{\ell+1} (a_{\ell+1}) \Gamma_{\ell+1}\cdots\Gamma_2.
\end{multline*}
Finally, we delete all occurrences of $a_\ell$.
The closure of the positive braid obtained in this way is the split union of the closures of the braids
$\Gamma_2\cdots\Gamma_{\ell}\Omega_\ell^{n-2\ell+1}\Gamma_\ell\cdots\Gamma_2$ and $\Delta_{n-2\ell+1}$.
This procedure is illustrated in \cref{fig:Deltatoaabb} for $n=10$ and $\ell=3$.
\end{proof}

\begin{proof}[Proof of \cref{fifth}]
Consider the positive braid word $\Delta_{5n}$. By \cref{subsurface} with $\ell=3$,
$\Sigma(\Delta_{5n})$ contains $$\Sigma(\Gamma_2\Gamma_{3}\Omega_3^{5n-5}\Gamma_3\Gamma_2) \sqcup \Sigma(\Delta_{5n-5})$$ as an incompressible subsurface.
Using \cref{subsurface} with $\ell=3$ inductively on the last split summand, we obtain that $\Sigma(\Delta_{5n})$ contains
$$\Sigma(\Gamma_2\Gamma_{3}\Omega_3^{5n-5}\Gamma_3\Gamma_2) \sqcup \cdots \sqcup \Sigma(\Gamma_2\Gamma_{3}\Omega_3^{5}\Gamma_3\Gamma_2)$$ as an incompressible subsurface.
The same argument gives
$$\Sigma(\Gamma_2\Gamma_{3}\Omega_3^{10n-10}\Gamma_3\Gamma_2) \sqcup \cdots \sqcup \Sigma(\Gamma_2\Gamma_{3}\Omega_3^{10}\Gamma_3\Gamma_2)$$ as an incompressible subsurface of the fibre surface $\Sigma(\Delta_{5n}^2)$.
By the definitions of $\Omega_3$, $\Gamma_2$ and $\Gamma_3$,
the positive braid $\Gamma_2\Gamma_{3}\Omega_3^{10i}\Gamma_3\Gamma_2$ contains $a_1(a_1^2a_2^2)^{10i}$ as a subword.
Furthermore the closure of the braid $a_1(a_1^2a_2^2)^{10i}$ has genus defect at least $5i$ (see \cref{kurzerwurm}).
In this way, using all the surfaces of the split union, we can produce a genus defect of at least
$$\sum_{i=1}^{n-1}5i %
=\frac{5n^2 -5n}{2}.$$
From this we obtain
$$\frac{\Delta g(T(5n,5n))}{g(T(5n,5n))} \geq \frac{{5n^2-5n}}{{25n^2-15n+2}}\xrightarrow{n\to\infty}\frac{1}{5},$$
which establishes \cref{fifth}.
\end{proof}

\begin{proof}[Proof of \cref{T2}]
 We proceed as in the proof of \cref{fifth}. However, instead of $\ell=3$ we use $\ell=5$ when applying \cref{subsurface}
 and obtain that $\Sigma(\Delta_{9n})$ contains $$\Sigma(\Omega_5^{9(n-1)}) \sqcup \Sigma(\Omega_5^{9(n-2)}) \sqcup \cdots \sqcup \Sigma(\Omega_5^{9})$$
 as an incompressible subsurface.
 As seen in \cref{langerwurm}, the closure of the braid $\Omega_5^{4+7j}$ has genus defect at least $4+8j$.
 For every split summand $\Omega_5^{9i}$, we consider the largest subword of the form $\Omega_5^{4+7j}$ and produce genus defect accordingly.
 In this way, we produce at least
 $4 + 8\left\lfloor \frac{9i-4}{7}\right\rfloor \geq \frac{72i}{7} - \frac{60}{7}$
 genus defect per summand.
 In total, this amounts to a genus defect of at least $$\sum_{i=1}^{n-1}\frac{72}{7}i -\frac{60}{7} =  \frac{72n^2}{14} + O(n).$$
 On the other hand, we have $$g(\Sigma(\Delta_{9n})) = \frac{81n^2}{4} + O(n).$$
 From this we obtain $$g_4(\Sigma(\Delta_{9n}))\le \frac{81n^2}{4} + O(n) -\frac{72n^2}{14} - O(n)= \frac{423n^2}{28} + O(n),$$
 which finally yields
 $$\frac{g_4(T(9n,9n))}{g(T(9n,9n))} \leq \frac{\frac{423n^2}{28} + O(n)}{\frac{81n^2}{4}+ O(n)}\xrightarrow{n\to\infty}\frac{47}{63}<\frac{3}{4}$$
 and establishes \cref{T2}.
\end{proof}

\section{Slice genus of small torus knots}
\label{sec4}
This section is devoted to the proof of \cref{T1}. In fact, we will 
prove a generalisation to links. For links, the topological slice genus is bounded by the signature and nullity (denoted by $\mu$) as follows \cite{KT}:
\[
|\sigma(L)| - \# L + 1 + \mu(L) \leq 2g_4(L).
\]
\begin{prop}\label{prop1}
Let $L= T(p,q)$ be a torus link with non-maximal signature and nullity bound,~i.e.
$L \neq T(2,n)$, $T(3,3)$, $T(3,4)$, $T(3,5)$, $T(3,6)$, $T(4,4)$. Then
$$g_4(L) \leq \frac{6}{7} \, g(L).$$
\end{prop}
According to \cref{T2}, most torus links satisfy $\frac{g_4}{g}<\frac{3}{4}$.
The bulk of the proof of \cref{prop1} is thus an investigation of small torus links.
Their genus defects can often be found by computer calculation (see \cref{comp}),
or are inherited by incompressible subsurfaces, e.g. using the following construction:
\begin{lemma}[{\cite[Proposition~1]{Ba1}}]
\label{abc}
Let $p,q,r \in \N$ with $p \leq r$. Then $\Sigma(T(pq,r))$ contains $\Sigma(T(p,qr))$ as incompressible subsurface.\hfill$\Box$
\end{lemma}

The following lemma helps us dealing with the exceptional cases in the proof of \cref{prop1}:
\begin{lemma}\label{lemma345}
The following lower bounds hold for the quotient $\frac{2\Delta g (T(p,q))}{b_1(T(p,q))}$:
\begin{enumerate}\renewcommand{\theenumi}{\roman{enumi}}
\item For $3 | p$ and $q \geq 10$, the quotient is greater or equal to $8 / 51$.
\item For $4 | p$ and $q \geq 7$,  the quotient is greater or equal to $2 / 11$.
\item For $5 | p$ and $q \geq 6$,  the quotient is greater or equal to $1 / 5$.
\end{enumerate}
\end{lemma}
\begin{table}[p]%
\captionsetup{width=\textwidth}
\newcommand{\computer}{\cref{comp}}
\small
\noindent\begin{tabu}{*{4}{|[lightgray]l}||*{4}{l|[lightgray]}} \tabucline[lightgray]{1-8}
$b_1$ & $(p,q)$ & $\Delta g$ & Lower bound       & $b_1$ & $(p,q)$ & $\Delta g$ & Lower bound \\\hline
4   & (3,3)   & 0        &                       & 40  & (5,11)  & [5,6]    & \computer                 \\
6   & (3,4)   & 0        &                       & 40  & (6,9)   & [4,6]    &                           \\
8   & (3,5)   & 0        &                       & 42  & (3,22)  & [4,6]    &                           \\
9   & (4,4)   & 0        &                       & 42  & (4,15)  & [4,6]    &                           \\
10  & (3,6)   & 0        &                       & 42  & (7,8)   & [5,6]    & \computer                 \\
12  & (3,7)   & 1        & \cref{t37}            & 44  & (3,23)  & [5,6]    & $\Sigma (3, 10) \sqcup \Sigma(3, 13)$  \\
12  & (4,5)   & 1        & \cref{t45}            & 44  & (5,12)  & [5,7]    &                           \\
14  & (3,8)   & 1        &                       & 45  & (4,16)  & [5,6]    & $\Sigma (4, 5) \sqcup \Sigma(4, 11)$   \\
15  & (4,6)   & [1,2]    &                       & 45  & (6,10)  & [5,8]    & $\Sigma (6\cdot 2,5)$       \\
16  & (3,9)   & 1        &                       & 46  & (3,24)  & [5,6]    &                           \\
16  & (5,5)   & 1        &                       & 48  & (3,25)  & [5,7]    &                           \\
18  & (3,10)  & 2        & \computer             & 48  & (4,17)  & [5,7]    &                           \\
18  & (4,7)   & 2        & \computer             & 48  & (5,13)  & [5,8]    &                           \\
20  & (3,11)  & 2        &                       & 48  & (7,9)   & [5,8]    &                           \\
20  & (5,6)   & 2        & $\Sigma (5\cdot 2,3)$ & 49  & (8,8)   & [5,6]    &                           \\
21  & (4,8)   & 2        &                       & 50  & (3,26)  & [6,7]    & $\Sigma (3, 13) \sqcup \Sigma(3, 13)$           \\
22  & (3,12)  & 2        &                       & 50  & (6,11)  & [5,8]    &                           \\
24  & (3,13)  & 3        & \computer             & 51  & (4,18)  & [6,8]    & $\Sigma (4, 7) \sqcup \Sigma(4, 11)$   \\
24  & (4,9)   & 3        & \computer             & 52  & (3,27)  & [6,7]    &                           \\
24  & (5,7)   & 3        & \computer             & 52  & (5,14)  & [6,8]    & $\Sigma (5, 6) \sqcup \Sigma(5, 8)$    \\
25  & (6,6)   & 2        &                       & 54  & (3,28)  & [6,8]    &                           \\
26  & (3,14)  & 3        &                       & 54  & (4,19)  & [6,8]    &                           \\
27  & (4,10)  & [3,4]    &                       & 54  & (7,10)  & [6,9]    & $\Sigma (4, 7) \sqcup \Sigma(6, 7)$    \\
28  & (3,15)  & 3        &                       & 55  & (6,12)  & [6,8]    & $\Sigma (5, 6) \sqcup \Sigma(6, 7)$    \\
28  & (5,8)   & 4        & \computer             & 56  & (3,29)  & [6,8]    &                           \\
30  & (3,16)  & [3,4]    &                       & 56  & (5,15)  & [7,10]    & $\Sigma (5, 7) \sqcup \Sigma(5, 8)$    \\
30  & (4,11)  & 4        & \computer             & 56  & (8,9)   & [6,9]    & $\Sigma (4, 9) \sqcup \Sigma(4, 9)$            \\
30  & (6,7)   & 4        & \computer             & 57  & (4,20)  & [7,8]    & $\Sigma (4, 9) \sqcup \Sigma(4, 11)$   \\
32  & (3,17)  & 4        & \computer             & 58  & (3,30)  & [7,8]    & $\Sigma (3, 13) \sqcup \Sigma(3, 17)$  \\
32  & (5,9)   & 4        &                       & 60  & (3,31)  & [7,9]    &                           \\
33  & (4,12)  & 4        &                       & 60  & (4,21)  & [7,9]    &                           \\
34  & (3,18)  & 4        &                       & 60  & (5,16)  & [8,10]   & $\Sigma (5, 8) \sqcup \Sigma(5, 8)$            \\
35  & (6,8)   & [4,6]    &                       & 60  & (6,13)  & [6,10]   & $\Sigma (3, 13) \sqcup \Sigma(3, 13)$           \\
36  & (3,19)  & [4,5]    &                       & 60  & (7,11)  & [7,10]   & $\Sigma (5, 7) \sqcup \Sigma(6, 7)$    \\
36  & (4,13)  & [4,5]    &                       & 62  & (3,32)  & [7,9]    &                           \\
36  & (5,10)  & 4        &                       & 63  & (4,22)  & [8,10]    & $\Sigma (4, 11) \sqcup \Sigma(4, 11)$           \\
36  & (7,7)   & [4,6]    &                       & 63  & (8,10)  & [8,12]   & $\Sigma (5, 8) \sqcup \Sigma(5, 8)$            \\
38  & (3,20)  & [4,5]    &                       & 64  & (3,33)  & [7,9]    &                           \\
39  & (4,14)  & [4,6]    &                       & 64  & (5,17)  & [8,11]   &                           \\
40  & (3,21)  & [4,5]    &                       & 64  & (9,9)   & [7,9]    & $\Sigma (4, 9) \sqcup \Sigma(5, 9)$    \\
\tabucline[lightgray]{1-8}
\end{tabu}
\medskip
\caption{\small All $(p,q)$-torus links with $p,q\geq 3$ up to Betti number $b_1 \leq 64$,
    including all links of genus $g \leq 28$.
    The upper bounds for the genus defect $\Delta g$ are induced by the signature and nullity functions.
    For the lower bounds, there is either a reference given, or an incompressible
    subsurface from which the defect is inherited (see \cref{inherit,abc}).
    Subsurfaces of the kind $\Sigma(p-r,q) \subset \Sigma(p,q)$ are left out.}
\label{table64}
\end{table}
\begin{proof}
To prove (i), let $p = 3a$. By \cref{abc}, we have
\[
\Delta g (T(p,q)) \geq \Delta g (T(3,aq)).
\]
Let $aq = 17k + r$ with $0 \leq r \leq 16$. Applying the computed defects shown in \cref{table64} of
the knots $T(3,7),T(3,10),T(3,13)$ and $T(3,17)$ yields
\[
\Delta g (T(3,aq)) \geq 4k + s(r),
\]
where $s(r) = 0, 1, 2, 3$ for $r$ in $[0,6],[7,9], [10,12], [13,16]$, respectively.
For
\begin{align*}
\frac{2\Delta g (T(3a,q))}{b_1(T(3a,q))} & \geq \frac{8}{51},\\
\intertext{it suffices that}
4k + s(r)            & \geq \frac{4}{51}(3a-1)(q-1) \quad\Leftrightarrow\\
51k + 51s(r)/4       & \geq 3aq - 3a - q + 1        \quad\Leftrightarrow\\
51k + 51s(r)/4       & \geq 51k + 3r - 3a - q + 1   \quad\Leftrightarrow\\
3a + q               & \geq 1 + 3r - 51s(r)/4       \quad\Leftarrow\\
\intertext{(to find the maximum of the right-hand side, which is at $r = 6$,
        it suffices to check the cases $r = 6,9,12,16$)}
3a + q               & \geq 19                      \quad\Leftrightarrow\\
\frac{(3a-1) + (q-1)}{2} & > 8.                        \\
\intertext{Since the arithmetic dominates the geometric mean, this is implied by}
\sqrt{(3a-1)(q-1)}   & > 8                          \quad\Leftrightarrow\\
\sqrt{b_1(T(3a,q))}  & > 8                          \quad\Leftrightarrow\\
b_1(T(3a,q))         & > 64.
\end{align*}
The case $b_1(T(3a, q)) \leq 64$ is dealt with by \cref{table64}.
The proofs of (ii) and (iii) proceed in the same way.
For (ii), let $aq = 11k + r$, and use the computed defects of $T(4,5),T(4,7),T(4,9)$ and $T(4,11)$.
This covers the case $b_1(T(4a,q)) > 49$.
For (iii), setting $aq = 8k + r$ and using $T(5,4),T(5,6),T(5,7), T(5,8)$ covers
the case $b_1(T(5a,q)) > 36$. 
\end{proof}
\begin{proof}[Proof of \cref{prop1}]
The cases $p, q\leq 9$ are all contained in \cref{table64}. So let us assume $q \geq 10$.
We will prove that in this case we even have $2\Delta g/b_1 \geq 1/7$, which suffices since $b_1 \geq 2g$.
If $p$ is divisible by $3$, $4$ or $5$, then the statement follows from \cref{lemma345}.
All other $p$ can be written as $p = 3a + 4b$ with $a, b \geq 1$.  By \cref{lemma345},
\begin{align*}
2\Delta g (T(3a + 4b, q)) & \geq 2\Delta g (T(3a,q)) + 2\Delta g (T(4b,q)) \\
                        & \geq \frac{8(3a-1)(q-1)}{51} + \frac{2(4b-1)(q-1)}{11}.
\end{align*}
So now it suffices to show
\begin{align*}
\frac{8(3a-1)(q-1)}{51} + \frac{2(4b-1)(q-1)}{11} & \geq \frac{(3a + 4b - 1)(q-1)}{7} \quad\Leftrightarrow \\
616(3a-1) + 714(4b-1)     & \geq 561(3a + 4b - 1) \quad\Leftrightarrow \\
1848a - 616 + 2856b - 714 & \geq 1683a + 2244b - 561 \quad\Leftrightarrow \\
165a + 612b                 & \geq 769,
\end{align*}
which follows from $a, b\geq 1$.
\end{proof}

\bibliographystyle{myamsalpha}
\bibliography{torus_def}
{\footnotesize

\bigskip
\textsc{Universit\"at Bern, Mathematisches Institut, Sidlerstrasse 5, 3012 Bern, Switzerland}

\myemail{sebastian.baader@math.unibe.ch}

\myemail{lukas.lewark@math.unibe.ch}

\myemail{livio.liechti@math.unibe.ch}

\bigskip
\textsc{Max Planck Institute for Mathematics, Vivatsgasse 7, 53111 Bonn, Germany}

\myemail{peter.feller@math.ch}
}
\end{document}